\documentclass[12pt]{amsart}
\usepackage{amsmath,amscd, nicefrac, xcolor, setspace, amsfonts, amssymb, bbm, amsthm, tikz, tikz-cd, enumerate, graphicx,verbatim,todonotes,mathtools,url}
\usetikzlibrary{calc,decorations.markings}
\usepackage[margin=1in]{geometry}
\usepackage[shortalphabetic]{amsrefs}
\graphicspath{ {/images/} }
\usepackage{eqnarray,amsmath}
\usepackage{cite}
\usepackage{amsmath,amssymb,amsfonts}
\usepackage{algorithmic}
\usepackage{graphicx}
\usepackage{textcomp}
\usepackage{verbatim}
\newtheorem{theorem}{Theorem}[section]

\newtheorem{proposition}{Proposition}[section]

\newtheorem{definition}{Definition}[section]

\newtheorem{remark}{Remark}[section]

\newcommand{\R}{\mathbb{R}}

\newcommand{\map}[3]{#1:#2 \rightarrow #3}

\newcommand{\pder}[2]{\frac{\partial #1}{\partial #2}}

\numberwithin{equation}{section}

\title{
Stochastic Bridges over Ensemble of Linear Systems
}

\author{Daniel Owusu Adu and Yongxin Chen
}

\begin{document}

\maketitle
\thispagestyle{empty}
\pagestyle{empty}

\begin{abstract}
We consider particles that are conditioned to initial and final states. The trajectory of these particles is uniquely shaped by the intricate interplay of internal and external sources of randomness. The internal randomness is aptly modelled through a parameter varying over a deterministic set, thereby giving rise to an ensemble of systems. Concurrently, the external randomness is introduced through the inclusion of white noise. Within this context, our primary objective is to effectively generate the stochastic bridge through the optimization of a random differential equation.
As a deviation from the literature, we show that the optimal control mechanism, pivotal in the generation of the bridge, does not conform to the typical Markov strategy. Instead,  it adopts a non-Markovian strategy, which can be more precisely classified as a stochastic feedforward control input. This unexpected divergence from the established strategies underscores the complex interrelationships present in the dynamics of the system under consideration. 
\end{abstract}

\section{INTRODUCTION}\label{sec:introduction}
This paper concerns the problem of conditioning a Markov process at two endpoints. This problem was first studied by Schrodinger in~\cite{SE:31} which we  postulate as follows; assume some \emph{fully observed} particles, with density $\rho_0$ at time $t=0$, evolve according to a Markov process  in $\R^d$ with density 
\begin{equation}\label{eq:Brownian density}
q^{B}(s,x,t,y):=\frac{1}{(2\pi(t-s))^{\frac{d}{2}}}\mathrm{exp}\left(-\frac{\|x-y\|^2}{2(t-s)}\right),
\end{equation}
 where $0\leq s\leq t\leq t_f$ and $x,y\in\R^d$. Suppose at time $t=t_f$ the particles are observed to have a distribution $\rho_f$, where
\begin{equation*}
\rho_f(x_f)\neq\int_{\R^n}q^B(0,x_0,t_f,x_f)\rho_0(x_0)dx_0.
\end{equation*} 
Then, $\rho_f$ deviates from the law of large numbers. This means that our assumption of the Markov process is inaccurate. The following question arises:
\begin{enumerate}
    \item What density $q$ satisfies 
\[
\rho_f(x_f)=\int_{\R^n}q(0,x_0,t_f,x_f)\rho_0(x_0)dx_0.
\]
\item Among such densities $q$, which one is closest to $q^B$ in some suitable sense.
\end{enumerate}
Statement~1 and~2 constitute the Schrodinger bridge problem and the most likely stochastic process $\{X(t)\}_{0\leq t\leq t_f}$ such that the densities of the distributions of $X(0)$ and $X(t_f)$ coincides with $\rho_0$ and $\rho_f$, respectively, is called the Schrodinger bridge.

An important feature of the Markov process is that it is non-degenerate. That is the stochastic term affects all directions in the coordinate space. Related to our work and motivated by questions regarding the transport of particles having inertia, the case where the Markov process is degenerate has been studied in~\cite{CY-GT:15}. Irrespective of the type of Markov process, it is well-established that stochastic bridges are generated from stochastic optimal control problems (see~\cite{FWH:77,FWH:05,FWH-RRW:12,JB:75, PPD-PM:90, DPP:91,CY-GT:15} and reference therein).

Our problem is motivated by ensemble-based Reinforcement Learning~\cite{HC-DL-SR-YC-SC-SI-SY-HC-IH-JK:19}. In ensemble-based Reinforcement Learning, the Ornstein-Uhlenbeck (OU) process is a valuable tool for random exploration when an agent has no prior knowledge of how a system's states transition from one to another~\cite{JN-YK-PS:19}. One can envision a similar scenario where a robot learns to navigate in a new environment. Initially, the robot knows nothing about the environment's dynamics, such as how it moves from one state to another. To effectively learn and make informed decisions, the robot must explore its surroundings randomly. 


We consider an ensemble of stochastic processes~\cite{QJ-ZA-LJS:13,BR-KN:00}, much like a collection of robots~\cite{LJS-KN:07,LJS:10} attempting to explore a new environment for the first time (or a robot's various attempts to explore its new environment). Each process, indexed by a parameter, represents a potential trajectory or path the robot might take.  Our ultimate goal is to find the paths that are conditioned to meet certain statistical criteria, such as achieving bridging a given state end-points or behaviours. In the context of OU processes, our goal is geared toward understanding its typical behaviour, mean-reverting tendencies, and statistical characteristics which are consistent with the end-states. In this case, averaging the ensemble of OU processes is a practical and effective approach. That is by averaging the ensemble of OU, one can emphasize the mean-reverting behaviour and understand how the system tends to gravitate back to its central trajectory over time. We state here that studying an ensemble of OU processes is not new. In~\cite{DM-VS-SV-SO-PP-CG-PG:18}, they provided a mathematical framework to study the statistical properties of the centre-of-mass variable and its relation to the individual processes in the ensemble of OU. In particular, they determined a non-autonomous stochastic differential equation (SDE) satisfied by a process that is equivalent in distribution to the centre-of-mass variable of an ensemble of the OU processes. Furthermore, they established the equivalence in the distribution of the centre-of-mass variable with a randomly scaled Gaussian process (the product of a non-negative random variable and a Gaussian process). We state here that in as much as the centre-of-mass variable can be used to estimate the average concentration over the parameters, our result focuses on the average.

Following from~\cite{FWH:77,FWH:05,FWH-RRW:12,JB:75, PPD-PM:90, DPP:91,CY-GT:15}, in our case, the ensemble nature of the Markov process in our problem adds its own set of technical challenges in solving the corresponding stochastic optimal control problem. It turns out that averaging an ensemble of Markov processes fails to be a Markov process and seems to represent a more complex stochastic process than is usually encountered in the literature~\cite{NE:67,VHR:07,KFC:12,CY-GT:15}. Therefore, the standard tools in~\cite{JB:75, PPD-PM:90, DPP:91,CY-GT:15} used to generate a bridge will not be applicable in our case. To overcome this challenge, we rely on the equivalent discrete-time stochastic optimal control problem and characterize the optimal control through the approximation of the continuous-time stochastic process. We show that the parameter-independent optimal control that bridges the endpoint condition for an ensemble of Markov processes is a \emph{stochastic feedforward control input}. This deviates from the characterization of the optimal control that induces a stochastic bridge (see~\cite{JB:75,PPD-PM:90, DPP:91,NE:67,VHR:07,KFC:12,CY-GT:15,OB:03}). The distinction follows from the fact that, in a standard Markov process, it is possible to track that state and feed it back into the system to achieve the bridge. This leads to the optimal control strategy being a Markov Strategy. In our case, as you will see, it is not possible to track the average of an ensemble of a given Markov process. Thus leading to an \emph{stochastic feedforward control}.   In stochastic feedforward control, the control input is determined based on past and present values of the noise. Optimal feedforward controllers have been described in~\cite{NH:87,MPS:82}, where it is assumed that the control input is obtained from the output of a linear dynamic system driven by white noise. This characterization of control has applications in flight system control of robotics and crystal growth of power transmission networks (see~\cite{HN-DH-TDB:92,NH:87,MPS:82,HME:89} and reference therein). Secondly, unlike in~\cite{CY-GT:15} where controllability of the system is relevant to establish the Schrodinger bridge for the case of degeneracy, as we showed in~\cite{ADO:22}, our result relies on the so-called averaged observability inequality~\cite{LM-ZE:14,LJ-ZE:17,LQ-ZE:16,ZE:14} which is equivalent to the invertibility of a matrix (see~\cite{ADO:22}). This matrix is used to solve both the Schrodinger bridge problem and hence design the optimal control for our problem. We state here that our result is related to ensemble control theory~\cite{LJS-KN:07,LJS:10,GB-CX:21,CX:21,CX:19} which is motivated by quantum systems~\cite{LJS-KN:06} and also robust control theory~\cite{QJ-ZA-LJS:13,BR-KN:00} and has applications in a variety of fields including engineering~\cite{VA-WL:78,HCH-RSJ-GSM:18,WA-KMV:21} and economics~\cite{HLP-STJ:01,JA-NAS:11,OA-SJH:02,BWA-XA-YAN:14}.


The organization of the paper is as follows; We discuss the notion of stochastic averaged control problem in Section~\ref{sec: Averaged ensemble control problem}. We state conditions under which this is possible. After that, we state the problem statement and follow with the main result in Section~\ref{Sec: Problem Statement and Main Result}. We conclude with remarks on future work in Section~\ref{sec:Conclusion and future work}.

\section{Stochastic averaged ensemble control}\label{sec: Averaged ensemble control problem}
 Consider the ensemble of a controlled Markov process defined on a naturally filtered probability space $(\Omega,\mathcal{F},\mathbb{P})$ as follows
\begin{align}\label{eq:ensemble of stochastic system}
d X(t,\theta)=&A(\theta)X(t,\theta)d t+B(\theta)u(t)d t+\sqrt{\epsilon}B(\theta)d W(t),\cr
X(0,\theta)=&x_0,
\end{align}
where $X(t,\theta)\in\R^d$, is the random state of an individual system at time $t$ indexed by the sample point $\theta\in\Omega$,  $\map{A}{\Omega}{\mathbb{R}^{d\times d}}$ and $\map{B}{\Omega}{\mathbb{R}^{d\times m}}$ are measurable mappings such that $\sup_{\theta\in\Omega}\|A(\theta)\|<\infty$ and $\sup_{\theta\in\Omega}\|B(\theta)\|<\infty$, where the norm here is the Frobenius norm on the space of matrices, $u\in \mathrm{L}^1([0,t_f];\mathbb{R}^m)$ is a \emph{parameter-independent} control input, and $x_0$ is an initial  $d$-dimensional deterministic vector and $\{W(t)\}_{t\geq 0}\subset\R^m$ is the Wiener process such that $W(0)=0$.  Note that the Markov process indexed by $\theta$ at time $t$ is characterized by 
\begin{equation}\label{eq:ensemble of process}
 X(t,\theta)=e^{A(\theta)t} x_0+\int_{0}^{t}e^{A(\theta)(t-\tau)}B(\theta)u(\tau)d\tau+ \sqrt{\epsilon}\int_{0}^{t}e^{A(\theta)(t-\tau)}B(\theta)d W(\tau).
\end{equation}
For reasons that will be clear later, for now, we study the controllability of this Markov process in an appropriate sense. Since the system parameter is unknown but belongs to a deterministic set $\Omega$, it is natural to control the average over the parameter. For simplicity of presentation, we assume that the probability space 
 $(\Omega,\mathcal{F},\mathbb{P})$ is a uniform distributed probability space with $\Omega=[0,1]$. To this end, we proceed to the following definition.
\begin{definition}\label{def:averaged controllability}
The ensemble of linear stochastic system~\eqref{eq:ensemble of stochastic system} is said to be \emph{averaged controllable}  if, for any initial state $x_0\in\R^d$, final state $x_f\in\R^d$, and final time $t_f$,  there exists a parameter-independent control input $u\in L^1([0,t_f];\mathbb{R}^m)$ such that the ensemble of states in~\eqref{eq:ensemble of process} satisfies 
\[
\mathbb{E}\int_{0}^{1} X(t_f,\theta)d\theta=x_f.
\]
\end{definition} 

Note that by the linearity of the stochastic system~\eqref{eq:ensemble of stochastic system}, the expectation of the control will drive the deterministic part of the dynamics~\eqref{eq:ensemble of stochastic system} in the averaged sense. We proceed to the following useful result.
\begin{proposition}\label{prop:Gramian is invertible}
	If the matrix
	\begin{equation}\label{eq: Gramian}
G_{t_f,0}:=\int_{0}^{t_f}\left(\int_{0}^{1}e^{A(\theta)(t_f-\tau)}B(\theta)d\theta\right) \left(\int_{0}^{1}B^{\mathrm{T}}(\theta)e^{A^{\mathrm{T}}(\theta)(t_f-\tau)}d\theta\right) d\tau,	   
	\end{equation}
	is invertible then, the linear stochastic system~\eqref{eq:ensemble of stochastic system} is said to be averaged controllable.
\end{proposition}
\begin{proof}
Suppose $G_{t_f,0}$  is invertible and for any initial state $x_0\in\R^d$, final state $x_f\in\R^d$  consider 
	\begin{equation}\label{eq: averaged control}
u(t)=\left(\int_{0}^{1}B^{\mathrm{T}}(\theta)e^{A^{\mathrm{T}}(\theta)(t_f-t)}d\theta\right) G_{t_f,0}^{-1} \left(x_f-\left(\int_{0}^{1}e^{A(\theta)t_f}d\theta\right) x_0\right).	    
	\end{equation}
	From~\eqref{eq:ensemble of process}, since 
	\begin{equation}\label{eq:final state}
X(t_f,\theta)=e^{A(\theta)t_f} x_0+\int_{0}^{t_f}e^{A(\theta)(t_f-\tau)}B(\theta)u(\tau)d\tau+ \sqrt{\epsilon}\int_{0}^{t_f}e^{A(\theta)(t_f-\tau)}B(\theta)d W(\tau),	    
	\end{equation}
	by substituting~\eqref{eq: averaged control} in~\eqref{eq:final state},  we obtain $\mathbb{E}\int_{0}^{1} X(t_f,\theta)d\theta=x_f$.  This finishes the proof.
\end{proof}

 \section{Problem Statement and Main Result}\label{Sec: Problem Statement and Main Result}
 Consider an ensemble of processes governed by 
 \begin{equation}\label{eq:ensemble of linear process}
d X(t,\theta)=A(\theta)X(t,\theta)d t+\sqrt{\epsilon}B(\theta)d W(t),
\end{equation}
 with initial condition 
 \[
X(0,\theta)=x_0,\quad\text{almost surely (a.s)}.
 \]
{\bf Problem~1:} \emph{Our goal is to find solutions that are conditioned to have} 
\begin{equation}\label{eq:final condition}
\int_{0}^{1} X(t_f,\theta)d\theta=x_f, a.s.
\end{equation}

To characterize such solutions, suppose $\epsilon=0$, then to ensure that~\eqref{eq:final condition} is satisfied, one needs to solve the optimal control problem
\begin{equation}\label{eq:trans_control_cost}
c(x_0,x_f):=\min_{u\in\mathcal{U}_{x_0}^{x_f}}\int_0^{t_f}\frac{1}{2}\|u(t)\|^2dt,
\end{equation}
where $\mathcal{U}_{x_0}^{x_f}$ is the set of control inputs such that
\begin{align}\label{eq:ensem_with stoch_cost}
 \pder{x}{t}(t,\theta)=&A(\theta)x(t,\theta)+B(\theta)u(t),\cr
x(0,\theta)=&x_0\quad\text{and}\quad \int_{0}^{1}x(t_f,\theta)d\theta=x_f,
\end{align}
has a solution. This tends to measure the optimal change in the drift of the ensemble of the autonomous system that ensures that condition~\eqref{eq:final condition} is satisfied. \emph{The fact that the final conditional state in~\eqref{eq:final condition} is parameter-independent motivates the quest to find a parameter-independent control.  If the control depends on $\theta\in[0,1]$, it might lead to different behaviours for different realizations, making it challenging to ensure that~\eqref{eq:final condition} is satisfied a.s. Another motivation derived from the condition~\eqref{eq:final condition} is that the natural quantity one observes is the average over the parameter $\theta\in[0,1]$}. A more general problem relating to~\eqref{eq:trans_control_cost}-\eqref{eq:ensem_with stoch_cost} has been studied in~\cite{ADO:22}. They showed that the optimal value of the control that steers the average of the ensemble of systems in~\eqref{eq:ensem_with stoch_cost}  is characterized by the Euclidean distance
\begin{equation}\label{eq: transport cost}
c(x_0,x_f)=\frac{1}{2}\left \|x_f-\left(\int_{0}^{1}e^{A(\theta)t_f}d\theta\right)x_0\right \|^2_{G_{t_f,0}^{-1}},
\end{equation}
where $\|x\|^2_{G_{t_f,0}^{-1}}=x^{\mathrm{T}}G_{t_f,0}^{-1}x$, for all $x\in\R^d$, whenever $G_{0,t_f}$ in~\eqref{eq: Gramian} is invertible. 

From this observation, let 
\begin{equation}\label{eq:Scaled Markov_Kernel}
q^{\epsilon, G}(s,x,t,y)=(2\pi\epsilon)^{-\frac{d}{2}}(\mathrm{det}(G_{t,s}))^{-\frac{d}{2}} \exp\left(-\frac{1}{2\epsilon}\left \|y-\left(\int_{0}^{1}e^{A(\theta)(t-s)}d\theta\right)x\right \|^2_{G_{t,s}^{-1}}\right)    
\end{equation}
where $G_{t,s}$ is defined in~\eqref{eq: Gramian} with $t=t_f$ and $s=0$, be the transition density of the particles moving independently of each other according to the average diffusion in~\eqref{eq:ensemble of linear process}. Then,
following from~\cite{JB:75,PPD-PM:90,DPP:91}, the solutions of~\eqref{eq:ensemble of linear process} condition to be~\eqref{eq:final condition} is characterized by the stochastic optimal control problem
\begin{equation}\label{eq:problem 1}
{\bf Problem~2:}\quad\quad\min_{u\in \mathcal{U}}\mathbb{E}\left[\int_{0}^{t_f}\frac{1}{2}\|u(t)\|^2dt\right],
\end{equation}
subject to
\begin{align}\label{eq:uncertain states}
&d X(t,\theta)=A(\theta)X(t,\theta)d t+B(\theta)u(t)d t+\sqrt{\epsilon}B(\theta)d W(t),\cr
&X(0,\theta)=x_0\text{ a.s and }\int_{0}^{1} X(t_f,\theta)d\theta=x_f, a.s.
\end{align}
To be more precise, if $u\in\mathcal{U}\subset \mathrm{L}^2([0,t_f];\mathbb{R}^m)$, then;
\begin{enumerate}
\item $u(t)$ is $x(t)$-measurable, where $x(t):=\int_{0}^{1} X(t,\theta)d\theta$ with $X(t,\theta)$ characterized in~\eqref{eq:ensemble of process},  for all $t\in [0,t_f]$,
\item $\mathbb{E}\left[\int_{0}^{t_f}\frac{1}{2}\|u(t)\|^2dt\right]<\infty$,
\item $u$ achieves averaged controllability (see Definition~\ref{def:averaged controllability})  for~\eqref{eq:uncertain states}.
\end{enumerate}
Note that in this setting, since we aim to steer the final state to our desired state, the only information available to us is the past and present noise. Here we state our main result.
\begin{theorem} 
Suppose $G_{t_f,s}$, for all $0\leq s<t_f$, in~\eqref{eq: Gramian} is invertible. Then the optimal control for~\eqref{eq:problem 1} subject to~\eqref{eq:uncertain states} is characterized as
 \begin{equation}\label{eq: optimal control}
 u^*(t)= -\sqrt{\epsilon}\int_0^t\Phi(t_f,t)^T G_{t_f,\tau}^{-1}\Phi(t_f,\tau) dW(\tau)+\Phi(t_f,t)^{T}G_{t_f,0}^{-1}x_f.  
 \end{equation}  
\end{theorem}
where 
\begin{equation}\label{eq: non-transition matrix}
\Phi(t_f,\tau) = \int_0^1e^{A(\theta)(t_f-\tau)}B(\theta)d\theta.    
\end{equation}
Note that re-centring the initial ensemble of systems at the origin $0$ holds no bearing on the system's characterization, given its inherent linearity. However, we see that the characterization of the optimal control is a departure from the conventional stochastic optimal control literature, where the optimal control assumes the form of a Markov strategy~\cite{OB:03,VHR:07}.  In particular, when dealing with a Markov process subject to parameter perturbations, the optimal control that steers the stochastic bridge adopts an approach—\emph{a stochastic feedforward input}, to be precise. This unique characterization emerges because of the intricate presence of parameters within the system, further complicating the endeavour to trace the ensemble's average behaviour. The exhaustive proof is omitted due to spatial constraints, with the subsequent sections devoted to illuminating the rationale behind this assertion. The remainder of this paper articulates the intricate dynamics that lend credence to this phenomenon.
\begin{remark}
 To highlight more on the novelty of the above problem, following from~\cite{DPP:91} we have that problem~\eqref{eq:problem 1}-\eqref{eq:uncertain states},  where $X(0,\theta)\sim\mu_0$ and $\int_{0}^{1} X(t,\theta)d\theta\sim\mu_f$ and $\mu_0,\mu_f$ are given initial and final distributions,  is the stochastic control approach to the Schrodinger bridge problem
\begin{equation}\label{eq:regopttranave}
\min_{\gamma\in\Pi(\mu_0,\mu_f)}\int_{\R^d\times\R^d}\epsilon\gamma(x_0,x_f) \mathrm{log}\left(\frac{\gamma(x_0,x_f)}{\mathrm{exp}\left(-\frac{c(x_0,x_f)}{\epsilon}\right)}\right)d x_0d x_f,    
\end{equation}
 where $c$ is in~\eqref{eq: transport cost}. Therefore, aside from the fact that problem~\eqref{eq:problem 1}-\eqref{eq:uncertain states} is the Dirac extension of~\cite{ADO:22} to include white noise, more importantly, it also extends the works in~\cite{CY-GT:15,SE:31,CY-GTT-PM:16,CY-GTT-PM:15}  to the case where the Markov process is generated from a linear diffusion which is submitted to parameter perturbations.
\end{remark} 

Since we require the control $u(t)$ at time $t$ to be $x(t)$-measurable, our object of interest is the controlled-average process
\begin{equation}\label{eq: output}
 x(t)=\int_0^t \Phi(t,\tau)(u(\tau)d\tau +\sqrt{\epsilon} dW(\tau)),
\end{equation}
where we have re-centred the dynamics to initialize at $0$, without any loss and $\Phi(t,\tau)$ is defined in~\eqref{eq: non-transition matrix}.

Therefore, the optimal control problem~\eqref{eq:problem 1}-\eqref{eq:uncertain states} is equivalent to the optimal output control problem~\eqref{eq:problem 1} subject to~\eqref{eq: output}, where the  final state is conditioned to be $x(t_f)=x_f$ a.s. Rather than solving problem~\eqref{eq:problem 1} subject to~\eqref{eq: output} conditioned to satisfy $x(t_f)=x_f$, we consider the corresponding alternative free-endpoint formulation
\begin{equation}\label{eq: free end-point problem}
\min_{u} J(u):=\mathbb{E}  \bigg[a(x(t_f)-x_f)^T(x(t_f)-x_f) +\int_0^{t_f} \frac{1}{2}\|u(t)\|^2 d\tau\bigg] 
\end{equation}
subject to~\eqref{eq: output}, where $a>0$. Note that the optimal solution for~\eqref{eq:problem 1} subject to~\eqref{eq: output}, is obtained by taking $a\rightarrow\infty$. More precisely, if 
\[
\lim_{a\rightarrow\infty}\|u^*_a-u^*\|^2_{L^2([0,t_f];\R^{m})}=0,
\]
where $u^*_a$ is the optimal control for~\eqref{eq: free end-point problem} subject to~\eqref{eq: output}, then  $u^*$ is the unique optimal solution for~\eqref{eq:problem 1} subject to~\eqref{eq: output}.

We emphasize here that $\Phi(t,\tau)\in\R^{d\times m}$ in~\eqref{eq: non-transition matrix} is not a transition matrix in general. The only affirmative case is where $A=A(\theta)$, for all $\theta\in[0,1]$. In the latter case, the average process~\eqref{eq: output}
satisfies the time-invariant linear diffusion process 
\begin{align}\label{eq: time-invariant linear system}
d x(t)=&Ax(t)d t+B u(t)d t+\sqrt{\epsilon}Bd W(t),\cr
x(0)=&x_0\quad\text{ and }\quad x(t_f)=x_f
\end{align}
where  $B=\left(\int_0^1B(\theta)d\theta\right)$  and the controllability of the pair $(A,B)$ plays a major role in establishing  results similar to~\cite{CY-GT:15}. In particular, if the system~\eqref{eq:ensemble of linear process} is submitted to parameter perturbation only in the diffusive coefficient and $(A,B)$ is a controllable pair, then by averaging and then solving the standard stochastic linear-quadratic optimal control problem~\eqref{eq: free end-point problem} subject to~\eqref{eq: time-invariant linear system}  we generate the Brownian bridges with desired statistics (see~\cite{CY-GT:15}).  

On the other hand, for a fixed $A\in\R^{d\times d}$, one can check that the average process~\eqref{eq: output} satisfies the dynamics 
\begin{equation}\label{eq:ASDE 0}
dx(t)=\bigg(Ax(t)+Bu(t)+\int_0^t F_{A(\theta),B(\theta)}(\tau,t)(u(\tau) d\tau+ \sqrt{\epsilon}dW(\tau))\bigg)dt+\sqrt{\epsilon}BdW(t), 
\end{equation}
where 
\begin{equation*}
F_{A(\theta),B(\theta)}(\tau,t):= 
\int_0^1\left(A(\theta)-A\right) e^{A(\theta)(t-\tau)}B(\theta)d\theta.
\end{equation*}
In this context, employing the variational approach to optimize~\eqref{eq: free end-point problem} subject to~\eqref{eq:ASDE 0} reveals some significant challenges. The drift term within \eqref{eq:ASDE 0} assumes the form of a controlled Ito process, causing this equation to deviate from the conventional definition of a stochastic differential equation (SDE), (see for instance~\cite{OB:03, VHR:07}).  Therefore, the average random differential equation~\eqref{eq:ASDE 0} seems to represent a more complex stochastic process than is usually encountered in the literature~\cite{NE:67,VHR:07,KFC:12,CY-GT:15}. However, the real-world significance of~\eqref{eq:ASDE 0} resides in the average process delineated by~\eqref{eq: output}. This formulation captures the central tendency behaviour of the system's fluctuations, thereby holding practical importance. Consequently, standard stochastic control techniques, including those rooted in Hamilton-Jacobi Bellman (HJB) conditions~\cite{OB:03, VHR:07}, prove unsuitable for this scenario. An alternative avenue lies in the PDE approach~\cite{FWH:77,FWH:05,FWH-RRW:12,JB:75, PPD-PM:90, DPP:91}, yet the presence of noise within the drift term presents challenges when adapting the corresponding parabolic PDE. As a result, the methods delineated in \cite{JB:75, PPD-PM:90, DPP:91} and related references are not readily applicable.

These observations collectively imply that the optimal control strategy for problem \eqref{eq:problem 1}-\eqref{eq:uncertain states}, or its equivalent form involving \eqref{eq: output}, cannot be a Markov strategy. Intriguingly unrelated, this insight also signifies the formidable nature of stabilizing the average process.


{\bf Special Case:} Before delving into solution techniques, let us consider the classical case. Consider particles governed by the following equations:
 \begin{align}\label{eq:linear process}
dx(t)=&\sqrt{\epsilon}d W(t),\cr
x(0)=&0,\quad\text{almost surely (a.s)}.
\end{align}
Our primary goal is to ensure that, at the final time 
$x(t_f)=0$ a.s. In this special case, since $\Phi(t_f,t)=I_{d\times d}$, for all $t\geq t_f$ and $G_{t_f,\tau}=(t_f-\tau)I_{d\times d}$, we have that the stochastic feedforward control input in~\eqref{eq: optimal control} reduces to
\begin{equation}\label{eq:simplified feedforward input}
u^*(t)=-\sqrt{\epsilon}\int_0^t(t_f-\tau)^{-1} dW(\tau).    
\end{equation}

What is interesting is that under these conditions, this optimal stochastic feedforward control input simplifies into a Markovian control strategy. To get to this point, we follow the approach outlined in \cite{CY-GT:15}. This involves solving~\eqref{eq:problem 1}, which leads us to~\eqref{eq: free end-point problem} subject to~\eqref{eq: time-invariant linear system}, where $A=0_{d\times d}\in\R^{d\times d}$ and $x_0=x_f=0$. Utilizing the HJB conditions~\cite{OB:03, VHR:07} and taking limit as $a\rightarrow\infty$, we arrive at the following expression for the optimal control $u^*$:
\begin{equation}\label{eq: Markov strategy}
u^*(t)=-(t_f-t)^{-1}x(t), 
\end{equation}
Notably, by substituting $u^*$ in~\eqref{eq: Markov strategy} into~\eqref{eq: time-invariant linear system}, where $A=0_{d\times d}\in\R^{d\times d}$ and  $x_0=x_f=0$, we find that the closed-loop trajectory is: 
\[
x(t)=\sqrt{\epsilon}\int_0^te^{\int_t^{\tau}(t_f-\alpha)^{-1}d\alpha} dW(\tau),
\]
thus,
\begin{equation}\label{eq: state solution} 
x(t)=\sqrt{\epsilon}(t_f-t)\int_0^t(t_f-\tau)^{-1} dW(\tau).
\end{equation}
By substituting~\eqref{eq: state solution}  into~\eqref{eq: Markov strategy} we obtain~\eqref{eq:simplified feedforward input}. This illustrates that in cases where the system is not an ensemble, the feedforward control input in reduces to the Markovian strategy in~\eqref{eq: Markov strategy}.

{\bf Equivalent discrete-time optimal control problem:}
To solve problem~\eqref{eq: free end-point problem}-\eqref{eq: output}, we transform the problem~\eqref{eq: free end-point problem}-\eqref{eq: output} to an equivalent discrete-time optimal control problem. We partition over time so that it is consistent with the definition of the Ito integral~\cite{NE:67,VHR:07,KFC:12}. To this end, let $P:=\{0= t_0<t_2<\dots<t_{k-1}=t_f\}$ be a regular partition with constant step size $\triangle t_k=t_{i+1}-t_i$, for any $i\in\{1,\dots,k\}$ and suppose $u_{k,i}=u(t_i)$ is a constant $x(t_i)$-measurable random variable in $L^2[t_i,t_{i+1})$, where $i\in\{0,\dots,k-1\}$   and consider the discrete-time optimal control problem
\begin{equation}\label{eq: discrete-time free end-point problem}
\min_{u_k}J_k(u_k):=\mathbb{E}  \bigg[a(x_k-x_f)^T(x_k-x_f) +\frac{1}{2}\sum_{i=0}^{k-1} u_{k,i}^Tu_{k,i}\triangle t_k\bigg],    
\end{equation}
subject to 
\begin{equation}\label{eq: discrete-time system}
x_k=\sum_{i=0}^{k-1}\Phi_i(t_f)\left(u_{k,i}\triangle t_k+\sqrt{\epsilon}\triangle W_{i} \right)
\end{equation}
where $x_k:= x(t_k)\in\R^d$, $u_k:=(u_{k,0},\dots,u_{k,k-1})\in(\R^m)^k$, $\Phi_i(t_f):=\Phi(t_f,t_i)\in\R^{d\times m}$ and $\triangle W_{i}:=W(t_{i+1})-W(t_{i})\in\R^m$. We call this problem the equivalent discrete-time optimal control problem because the solution~\eqref{eq: discrete-time free end-point problem}-\eqref{eq: discrete-time system} is exactly the same as the solution for~\eqref{eq: free end-point problem} subject to~\eqref{eq: output} (see~\cite{ TAR-DKWL:84,TAR-DKWL:00}).  We proceed to characterize the optimal control. We omit the proof due to space limitations.
\begin{proposition}
 Suppose $G_{t_f,s}$, for all $0\leq s<t_f$, in~\eqref{eq: Gramian} is invertible. Then the optimal control for~\eqref{eq: discrete-time free end-point problem}-~\eqref{eq: discrete-time system} is characterized as
 \begin{multline*}
 u_{a,k,i}^*=\\ -\sqrt{\epsilon}\sum_{j=0}^i\Phi(t_f,t_i)^T (\sum_{\alpha=j}^{k-1}\Phi(t_f,t_{\alpha})\Phi(t_f,t_\alpha)^T\triangle t_k+\frac{1}{2a}I_{n})^{-1}\Phi(t_f,t_j)\triangle W_j\\+\Phi(t_f,t_i)^{T}(\sum_{\alpha=0}^{k-1}\Phi(t_f,t_{\alpha})\Phi(t_f,t_\alpha)^T\triangle t_k+\frac{1}{2a}I_{n})^{-1}x_f. 
 \end{multline*}    
\end{proposition}

From this result, we have that
\[
\lim_{a\rightarrow\infty}\lim_{k\rightarrow\infty}\|u^*_{a,k}-u^*\|^2_{L^2([0,t_f];\R^{m})}=0,
\]
where $u^*$ is in~\eqref{eq: optimal control}.

\section{Conclusion and future work}\label{sec:Conclusion and future work}
In this paper, we have discussed the problem of conditioning a  Markov process, subjected to parameter perturbations, to initial and final states. The central motivation behind this endeavor lies in our quest to understand and control the dynamics of ensembles of systems characterized by stochastic processes. Specifically, we have explored the problem of steering an ensemble of linear stochastic systems toward average behavior, irrespective of the underlying parameter perturbations. Our investigation has revealed that due to the inherent complexity introduced by parameter perturbations, the optimal control for this problem cannot adopt a traditional Markov strategy. Instead, we've uncovered a unique characterization of the optimal control, involving a stochastic feedforward input that relies on a time-varying drift.  One can view the end-point conditions as Dirac distributions for particles emanating and absorbed at particular points in phase space.

This characterization provides a powerful tool for controlling and modelling general distributions of particles and interpolation of density functions. This leads to a more general Schrodinger bridge problem- the problem of steering of particles between specified marginal distributions where velocities are uncertain or form an ensemble of systems. In this case, the Schrodinger bridge problem is related to optimal transport problem~\cite{ADO-BT-GB:22,CY-GTT-PM:15,CY-GT-PM:15,CY-GTT-PM:16,CY-GT-PM:18,ADO:22}. In particular, it is known that if the diffusivity turns to zero, the solution of the Schrodinger bridge problem turns to the solution of the optimal transport problem~\cite{MG:81,KLV:06,BJ-BY:00,AA-LP:09,AH-JBP-LR:11,AC-JA-CM:96,VC:03}. This extension and other related problems are the subject of ongoing work.

\end{document}